\documentclass[12pt]{amsart}
\usepackage{geometry,amsthm,amsmath,amssymb, graphicx, float, enumerate}
\usepackage[pdf]{pstricks}
\usepackage{epstopdf}
\restylefloat{table}
\restylefloat{figure}

\swapnumbers
\newtheorem{thm}{Theorem}[section]

\newtheorem{cor}[thm]{Corollary}

\newtheorem{lemma}[thm]{Lemma}
\newtheorem{prop}[thm]{Proposition}

\theoremstyle{definition}
\newtheorem{defn}[thm]{Definition}
\theoremstyle{remark}

\newtheorem*{question*}{Question}
\newtheorem*{answer*}{Answer}

\newcommand{\bb}[1]{\mathbb{#1}}

\def\eps{\epsilon}
\def\del{\delta}

\def\S1{\mathbb S}

\def\Snm1{\mathbb S^{n-1}}

\title{Stable phase retrieval with low-redundancy frames}
\date{}
\author{Bernhard G. Bodmann}
\address{651 Philip G. Hoffman Hall, Mathematics Department, University of Houston, Houston, TX 77204-3008}
\author{Nathaniel Hammen}
\address{651 Philip G. Hoffman Hall, Mathematics Department, University of Houston, Houston, TX 77204-3008}

\begin{document}
\maketitle

\begin{abstract}
We investigate the recovery of vectors from magnitudes of frame coefficients when the frames have a low redundancy,
meaning a small number of frame vectors compared to the dimension of the Hilbert space. We first show that for vectors in $d$ dimensions,
$4d-4$ suitably chosen frame vectors are sufficient to uniquely determine each signal, up to an overall unimodular constant,
from the magnitudes of its frame coefficients. Then we discuss the effect of noise and
show that 
$8d-4$ frame vectors provide a stable recovery if part of the frame coefficients is bounded away from zero. 
In this regime, perturbing the magnitudes of the frame coefficients by noise that is sufficiently small
results in a recovery error that is at most proportional to the noise level.   
\end{abstract}

\section{Introduction}

Phase retrieval is a topic that is currently extensively researched. Part of the effort is directed towards
applications in X-ray crystallography, where the Fourier transform dictates the form of the measured quantities
from which a signal is recovered \cite{Fien,BunkDPDSSV:07}. It is well known that the magnitudes
of the Fourier transform need to be complemented by additional information about the signal to 
make recovery feasible \cite{Aku}, for example the magnitudes of the fractional Fourier transform \cite{Jaming:10}.
It is then a challenge whether the additional measurements can be realized experimentally.
Another main motivation for phase retrieval is quantum communication, where quantum states need to
be estimated from the relative frequencies of outcomes occurring in quantum measurements \cite{Finkelstein:04}. In this paper,
we investigate the abstract question of finding a small number
of linear measurements such that their magnitudes characterize 
 a vector in a finite dimensional Hilbert space,
up to an overall unimodular constant.
In addition, we wish to make the recovery procedure resilient against noise affecting the magnitude measurements.
The central idea is that the redundancy inherent in the frame coefficients,
resulting from the linear dependencies among the frame vectors, compensates the loss of information when
passing from frame coefficients to their magnitudes. In fact, in frame theory the notion of redundancy
is usually understood to be the number of frame vectors divided by the dimension of the Hilbert space. In this paper, we choose the 
frame vectors in a specific way to recover the vectors, up to a unimodular constant.
We address the following main questions:
 How small can we choose the size of a frame and still characterize each vector uniquely?
What conditions ensure that the vector can be recovered with a guaranteed accuracy if the measured
magnitudes are affected by noise? 

Several strategies have been applied to the problem of phase retrieval, for example the reformulation of recovery in terms
of the rank-one hermitian $x \otimes x^*$ \cite{BalanCE:06,BalanBCE:07,BalanBCE:09}. This was solved with techniques from compressed sensing
by rank minimization in an underdetermined system \cite{CandesSV:11,CandesESV:11,
CandesL:12,DemanetH:12,EldarM:12}, or even without rank minimization \cite{WaldspurgerAM:12}.
However, this technique does not specify what redundancy, i.e.\ the 
is sufficient for recovery. Other recovery procedures embed in even higher dimensional spaces by taking
more tensor powers of $x$ with itself \cite{Balan:12}.
Constructions based on expander graphs and the polarization identity
give us randomized constructions with an explicit bound on the redundancy \cite{Bandeira}. However, this is still far from
the necessary number of vectors derived from the theory of projective embeddings \cite{HeinosaariMW:11},
see also \cite{Milgram:67,Steer:70,Mukherjee:81}.

The recovery procedure outlined here assumes that the signal is realized as a complex polynomial of degree $d-1$. 
In the absence of noise, the recovery proceeds in several steps:
\begin{enumerate}
\item The measured quantities are $\{|f(\omega_j)|^2, |f(\omega_j) + \nu_l f'(\omega_j) |^2 : 0 \le j \le 2n-2, 0 \le j \le 2\}$
where $\omega_j = e^{2\pi i j/(2n-3)}$ and $\nu_l = e^{2\pi i l/3}$. Using the Dirichlet kernel and the polarization identity, these are extrapolated
to the values $|f(z)|^2$ and $f'(z) \overline{f(z)}$ for each $z$ in the unit circle.
\item From the values of these two functions on the unit circle we determine moments of the roots of $f$.
\item The moments determine the polynomial up to an overall unimodular constant.
\end{enumerate}
We investigate how the presence of noise affects each of these steps, and show that under certain conditions, for all sufficiently small
$\eps$, perturbing the measured quantities up to $\eps$ still gives an approximate recovery with an error of order $\epsilon$. 
This provable stability only extends up to a certain noise level. Numerical experiments show that the domain in which the linear
error bound holds extends far beyond this level.

This paper is organized as follows: After fixing the notation, we show in Section~\ref{sec:inj} that a vector, up to a unimodular constant,
is determined by a specific choice of $4d-4$ measurements. Section~\ref{sec:stab} establishes criteria for stability of the recovery procedure,
complemented with the results of a numerical simulation.

\section{Injectivity of the magnitude map}\label{sec:inj}


\begin{defn}
The space of complex polynomials of degree at most $n$ is denoted as $P_n$.
It is equipped with the inner product induced by the Lebesgue measure on the unit circle,
so $p, q \in P_n$ have the inner product
$$
  \langle p, q \rangle = \int_{[0, 2\pi]} p(e^{it}) \overline{q(e^{it})} dt
$$
where the overbar denotes complex conjugation. The space of trigonometric polynomials
of degree at most $n$, henceforth called $T_n$, is understood to consist of all linear combinations of complex polynomials 
and of their complex conjugates. 
\end{defn}

Thus,  $P_n$ is the subspace of analytic functions in $T_n$. On the other hand,
the map $A:f \mapsto |f|^2$ Takes $f \in P_n$ to a trigonometric polynomial in $T_n$.
The first question we wish to resolve is at how many points $A(f)$ needs to be 
evaluated in order to determine $\{ \lambda f: |\lambda| = 1\}$. The second problem
is that of noisy recovery. If the measured quantities are perturbed, is it possible
to estimate the set accurately?

The following theorem shows that to determine a polynomial of degree at most $n-1$ up to a unimodular constant, is enough to know its magnitude 
at $4n-4$ points in the complex plane. An essential ingredient is a result of Philippe Jaming's \cite{Jaming:10}, here the special case for polynomials.

\begin{lemma} Let $d \in \mathbb N$.
If $g\in P_{d-1}$ then it is determined up to a unimodular constant by the values of $|g|^2$ on two lines $\mathbb L$ and $\mathbb L_\alpha$ that
intersect in an angle $\alpha \in \mathbb R \setminus \pi \mathbb Q$.
\end{lemma}
\begin{proof}
Without loss of generality we take $\mathbb L = \mathbb R$ and $\mathbb L_\alpha = e^{i \alpha} \mathbb R$. Then by the positivity of $|g|^2$ on $\mathbb R$, it extends to a polynomial whose roots come in pairs related by complex conjugation. On the other hand, the same applies to the extension of $|g|^2$ from $\mathbb L_\alpha$
to the complex plane and reflections of its roots about $\mathbb L_\alpha$.
Now we have a selection principle based on the pairings: We pick as the roots of $g$ the intersection of the sets of roots obtained from the two extensions.
If $g$ were not determined by this, then it would need to have a root in the symmetric difference of the roots from the two extensions. However, the
symmetric difference is invariant under reflecting first about $\mathbb L$ and then about $\mathbb L_\alpha$. This composition of the two reflections is an irrational rotation, and so if the symmetric difference is non-empty, it gives a dense set of roots in a circle, which means $g=0$.
\end{proof}

\begin{thm} Let $f(z) = \sum_{j=0}^{d-1} c_j z^j$, let $\mathbb S$ be the unit circle as before and $\mathbb S_\alpha = \phi_\alpha(\mathbb R) \cup \{1\}$
with $\alpha \in \mathbb R \setminus \pi \mathbb Q$ and $\phi_\alpha (z) =  \frac{e^{i\alpha} z- \omega}{e^{i\alpha}z-1}$.
Then sampling $\{|f(z^{(\alpha)}_j)|^2\}$ on $2d-1$ equidistantly spaced points $\{z^{(\alpha)}_j\}_{j=0}^{2d-2}$ of $\mathbb S_\alpha$
and $\{|f(\omega_l)|^2\}_{l=2}^{2d-2}$ determines $f$ uniquely, up to an overall unimodular factor.
\end{thm}
\begin{proof}
We proceed in several steps:

{\it Step 1.}
Given $f \in T_{d-1}$, $\omega = e^{2\pi i / (2d-1)} $ and the normalized Dirchlet kernel $D_{d-1}(z) = \frac{1}{2d-1} \sum_{j=-d+1}^{d-1} z^{j}$,
then $f(z) = \sum_{j=0}^{2d-2} f(\omega^j) D(z\omega^{-j})$. By substitution, for $a\in \mathbb C$, $r>0$, 
$f(z) = \sum_{j=0}^{2d-2} f(a+r\omega^j) D((z-a)\omega^{-j}/r)$.
This means from the values at $2d-1$ equidistant points on a circle we can interpolate any trigonometric polynomial of degree at most $d-1$.
Consequently, the magnitudes of $|f(z)|^2$ on $2d-1$ equidistant points on $\mathbb S_\alpha$ determine $|f(z)|^2$ on the entire circle. 
Because the circles $\mathbb S$ and $\mathbb S_\alpha$ intersect in $0$ and $\omega$,  this also determines the magnitudes
of $f$ at two of the sample points on $\mathbb S$. Once the additional magnitudes $\{|f(\omega_l)|^2\}_{l=2}^{2d-2}$ are obtained,
the Dirichlet kernel determines the magnitude of $|f|^2$ on all points in $\mathbb S \cup \mathbb S_\alpha$.

{\it Step 2.}
Using the Cayley map $z \mapsto  \frac{1+z}{1-z}$ and the associated polynomial
automorphism 
$$
   W f(z) = (1+z)^{k-1} f\Bigl(  -\frac{1- z}{1+z} \Bigr)
$$
we map both $\mathbb S$ and $\mathbb S_\alpha$ to lines $\mathbb L$ and $\mathbb L_\alpha$, $f$ to a polynomial $Wf$, and $|f|^2$
to a trigonometric polynomial $|Wf|^2$. By conformality, the angle between the lines at $\frac{1+\omega}{1-\omega}$ 
is
the same as the angle between the circles. However, the tangent vector $\phi_\alpha'(0)=-(1+\omega)e^{i\alpha}$ has an irrational argument, 
whereas $i\omega$ does not, so the two circles intersect with an angle that is an irrational multiple of $\pi$. By the conformality of the map,
the same is true for the intersection of $\mathbb L$ and $\mathbb L_\alpha$. 

{\it Step 3.}
Next, we use P. Jaming's argument for $Wf$ to show that magnitudes on the two lines 
uniquely determines $Wf$, up to a unimodular multiplicative constant \cite{Jaming:10}. 
Applying the inverse of the map $W$, the same applies to $f$.
\end{proof}

\section{Stable recovery in the presence of noise}\label{sec:stab}

The recovery procedure we outline below heavily relies on the analyticity properties of the function space. 
Although an absolute phase can not be measured, we have access to a relative phase
such as $f'(z)/f(z)$ for some $z \in \mathbb C$. If these values were known on the entire unit circle $ \S1=\{z \in \mathbb C: |z| =1\}$ then we could recover
$f$, up to an overall multiplicative constant from contour integrals and Cauchy formulas as outliend further below. 

We first examine how such integrals are affected by perturbed measurements. Unless noted otherwise, for any continuous $f: \mathbb C \to \mathbb C$,
$\|f\|_{\infty} = \max_{z \in\S1} |f(z)|$.

\begin{lemma}
Let $f: \mathbb C \to \mathbb C$ be an analytic function, and let $p_1:\S1\rightarrow\bb C$  and $p_2:\S1\rightarrow \mathbb R$ 
with $\|p_1\|_\infty<\eps$ and $\|p_2\|_\infty<\eps$. If there exists a $\del>0$ such that $(|f|^2+p_2)(z)>\del$ for all $z \in \S1$, then
$$\left\|\frac{f'\overline{f}+p_1}{|f|^2+p_2}-\frac{f'\overline{f}}{|f|^2}\right\|_\infty<\frac{\eps}{\del}\left(1+\left\|\frac{f'}{f}\right\|_\infty\right) \, .$$
\end{lemma}
\begin{proof}
\begin{align*}
\left\|\frac{f'\overline{f}+p_1}{|f|^2+p_2}-\frac{f'\overline{f}}{|f|^2}\right\|_\infty
&=\left\|\frac{|f|^2p_1-f'\overline{f}p_2}{(|f|^2+p_2)|f|^2}\right\|_\infty\\
&\le\left\|\frac{|f|^2p_1}{(|f|^2+p_2)|f|^2}\right\|_\infty+\left\|\frac{f'\overline{f}p_2}{(|f|^2+p_2)|f|^2}\right\|_\infty\\
&=\left\|\frac{p_1}{|f|^2+p_2}\right\|_\infty+\left\|\frac{p_2}{|f|^2+p_2}\frac{f'\overline{f}}{|f|^2}\right\|_\infty\\
&<\left\|\frac{\eps}{\del}\right\|_\infty+\left\|\frac{\eps}{\del}\frac{f'\overline{f}}{f\overline{f}}\right\|_\infty
=\frac{\eps}{\del}\left(1+\left\|\frac{f'}{f}\right\|_\infty\right)
\end{align*}
\end{proof}

Intrinsically the recovery procedure is linked to the moments of the roots of the polynomial inside the unit disk.
We study how Newton's identities are affected when the moments are perturbed.  
\begin{lemma}
Let $f$ be a complex polynomial with $N_0$ roots $\{z_j\}_{j=1}^{N_0}$ in the open unit disk. Let $f_i(z)=\sum_{k=0}^{N_0}b_kz^k=\prod_{j=1}^{N_0}(z-z_j)$ define the monic factor of $f$ whose roots are precisely the roots of $f$ that are inside the unit disk. Given the perturbed moments $\{\tilde \mu_k\}_{k=1}^{N_0}$ of the roots such that $|\tilde\mu_k-\mu_k|<\gamma$
for some $0 \le \gamma\le 1$ 
and $\mu_k=\sum_{j=1}^{N_0}z_j^k$ for all $k\in \{1, 2, \dots, N_0\}$ then there exists $C$ which only depends on $N_0$
such that $\{\tilde \mu_k\}_{k=1}^{N_0}$ uniquely determine
coefficients $\{\tilde{b}_k\}$ with  $|\tilde{b}_k-b_k| \le C \gamma $ for all $k \in \{1, 2, \dots, N_0\}$.
\end{lemma}
\begin{proof}
If we knew the values of $\mu_k$ we could recover the actual coefficients using Newton's identities, which give the recurrence relation $b_{N_0-k}=-\frac{1}{k}\sum_{l=1}^k\mu_l b_{N_0-k+l}$ for all $k$ from $1$ to $N_0$. Instead we use our approximations $\tilde\mu_k$ to find approximated coefficients $\tilde{b}_k$ using the recurrence relation $\tilde{b}_{N_0-k}=-\frac{1}{k}\sum_{l=1}^k\tilde\mu_l \tilde{b}_{N_0-k+l}$ with $\tilde{b}_{N_0}=1$. We inductively show that $|\tilde{b}_k-b_k|$ is $O(\gamma)$. For the base case, by assumption
$$\left|\tilde{b}_{N_0-1}-b_{N_0-1}\right|=\left|\tilde\mu_1b_{N_0}-\mu_1b_{N_0}\right|=|\tilde\mu_1-\mu_1|<\gamma \, .$$
For the inductive step we note that $|\mu_l|=|\sum_{j=1}^{N_0}z_j^l|\le\sum_{j=1}^{N_0}|z_j^l|\le N_0$, and if $S_k$ is the set of all combinations of $k$ roots of $f(z)$ inside the unit disk, then
$$|b_{N_0-k}|=\left|\sum_{S\in S_k}\prod_{z_j\in S}z_j\right|\le\sum_{S\in S_k}\prod_{z_j\in S}|z_j|\le\sum_{S\in S_k}1={{N_0}\choose{k}}$$
Thus, with the inductive assumption that for all $j<k$, there exists a constant $C_j$ such that $\left|\tilde{b}_{N_0-j}-b_{N_0-j}\right|<C_j\gamma$, we have
\begin{align*}
\left|\tilde{b}_{N_0-k}-b_{N_0-k}\right|
&=\frac{1}{k}\left|\sum_{l=1}^k\tilde\mu_l \tilde{b}_{N_0-k+l}-\sum_{l=1}^k\mu_l b_{N_0-k+l}\right|\\
&\le\frac{1}{k}\sum_{l=1}^k\left|\tilde\mu_l \tilde{b}_{N_0-k+l}-\mu_l b_{N_0-k+l}\right|\\
&=\frac{1}{k}\sum_{l=1}^k\left|\tilde\mu_l \tilde{b}_{N_0-k+l}-\tilde\mu_l b_{N_0-k+l}+\tilde\mu_l b_{N_0-k+l}-\mu_l b_{N_0-k+l}\right|\\
&\le\frac{1}{k}\sum_{l=1}^k\left(\left|\tilde\mu_l \tilde{b}_{N_0-k+l}-\tilde\mu_l b_{N_0-k+l}\right|+\left|\tilde\mu_l b_{N_0-k+l}-\mu_l b_{N_0-k+l}\right|\right)\\
&\le\frac{1}{k}\sum_{i=1}^k\left(|\tilde\mu_i|C_{k-i}\gamma+\left|b_{N_0-k+i}\right|\gamma\right)\\
&\le\frac{1}{k}\sum_{i=1}^k\left((|\mu_i|+\gamma)C_{k-i}\gamma+\left|b_{N_0-k+i}\right|\gamma\right)\\
&\le\frac{1}{k}\sum_{i=1}^k\left((N_0+1)C_{k-i}\gamma+{{N_0}\choose{k-i}}\gamma\right) \, ,
\end{align*}
thus $C_{k} = (1/k) \sum_{i=1}^k ((N_0+1) C_{k-i} + \left( {N_0 \atop k-i } \right))$ suffices.
\end{proof}

Next, we show how the coefficients of the monic polynomial factor containing the roots on the inside of the disk 
can be estimated from perturbed moments. 

\begin{thm}
Let $f(z)=\sum_{k=0}^Na_kz^k$, with fixed positive constants $m$ and $M'$ such that $0<m\le|f(z)|$ and $|f'(z)|\le M'$ for all $z$ on the unit circle $\S1$. Let 
$$\alpha=\frac{1}{1+2\left(1+\frac{M'}{m}\right)} \, ,$$
and $\eps>0$ with $\eps<\alpha m^2$, $p_1:\S1\rightarrow\bb C$ and $p_2:\S1\rightarrow\mathbb R$ with $\|p_1\|_\infty<\eps$, $\|p_2\|_\infty<\eps$, then
$$\tilde\mu_k=\frac{1}{2\pi i}\oint_{\S1}z^k\frac{f'\overline{f}+p_1}{|f|^2+p_2}dz $$
for $k \in \{1, 2, \dots, N_0\}$
observes
$$ | \mu_k - \tilde \mu_k | \le\frac{\eps}{(1-\alpha)m^2}\left(1+\frac{M'}{m}\right) \equiv \gamma
$$  
and if $\gamma \le 1$ then there exists $C$ which only depends on $N_0$ such that
$f_i(z)=\sum_{k=0}^{N_0}b_kz^k=\prod_{j=1}^{N_0}(z-z_j)$, the monic factor of $f$ whose roots are precisely the roots of $f$ that inside the unit disk,
has approximate coefficients $\{ \tilde b_k \}$ with 
$$
 |\tilde{b}_k-b_k| \le C \frac{\eps}{(1-\alpha)m^2}\left(1+\frac{M'}{m}\right) \, .
$$   
\end{thm}
\begin{proof}
Note that
$$\forall z\in \S1,\,(|f|^2+p_2)(z)\ge m^2+p_2(z)\ge m^2-\eps>(1-\alpha)m^2$$
so by the first lemma we have
$$\left\|\frac{f'\overline{f}+p_1}{|f|^2+p_2}-\frac{f'\overline{f}}{|f|^2}\right\|_\infty
<\frac{\eps}{(1-\alpha)m^2}\left(1+\left\|\frac{f'}{f}\right\|_\infty\right)
\le\frac{\eps}{(1-\alpha)m^2}\left(1+\frac{M'}{m}\right)$$
If we let $N_0$ be the number of roots of $f$ inside the unit circle, and we let $\mu_k=\sum_{j=1}^{N_0}z_j^k$, the $k$th moment of the inner roots of $f$, then the residue theorem gives us that for any integer $k\in[0,N]$
$$\mu_k=\frac{1}{2\pi i}\oint_{\S1}z^k\frac{f'\overline{f}}{|f|^2}dz$$
If we let
$$\tilde\mu_k=\frac{1}{2\pi i}\oint_{\S1}z^k\frac{f'\overline{f}+p_1}{|f|^2+p_2}dz=\mu_k+\frac{1}{2\pi i}\oint_{\S1}z^k\left(\frac{f'\overline{f}+p_1}{|f|^2+p_2}-\frac{f'\overline{f}}{|f|^2}\right)dz$$
then
\begin{align*}
|\tilde\mu_k-\mu_k|
=\frac{1}{2\pi} \left|\oint_{\S1}z^k\left(\frac{f'\overline{f}+p_1}{|f|^2+p_2}-\frac{f'\overline{f}}{|f|^2}\right)dz\right|
&\le\frac{1}{2\pi }\oint_{\S1}|z|^k\left|\frac{f'\overline{f}+p_1}{|f|^2+p_2}-\frac{f'\overline{f}}{|f|^2}\right| |dz|\\
&\le \frac{\eps}{(1-\alpha)m^2}\left(1+\frac{M'}{m}\right)
\end{align*}
Note that $N_0=\mu_0$. Because $\eps<\alpha m^2$ we have that
$$|\tilde\mu_0-\mu_0|
\le\frac{\eps}{(1-\alpha)m^2}\left(1+\frac{M'}{m}\right)
<\frac{\alpha m^2}{(1-\alpha)m^2}\left(1+\frac{M'}{m}\right)
=\frac{1}{\left(\frac{1}{\alpha}-1\right)}\left(1+\frac{M'}{m}\right)
=\frac{1}{2}$$
so rounding $\tilde\mu_0$ gives us $N_0$. Thus by the second lemma, we can recover an approximation for $f_i(z)$ with approximated coefficients $\tilde{b}_k$ such that 
$|\tilde{b}_k-b_k| \le C \gamma$. Now re-expressing $\gamma$ in terms of $\eps$ gives the desired result.
\end{proof}

\begin{cor}
If $f$ satisfies the hypotheses of the previous theorem, and in addition $|f(z)| \le M$ for $z \in \S1$, $\eps<\frac{\beta m^2}{d}$ for $d>N$ and
$$\beta=\frac{1}{1+2\left(1+\frac{(d-1)M+M'}{m}\right)}$$
and if $g(z)=z^{d-1}f(\frac{1}{z})$, then using the perturbation with $p_1$ and $p_2$ as above, we can recover an approximation for $g_i(z)=\sum_{k=0}^{N_0}b_kz^k$ (the monic factor of $g(z)$ whose roots are precisely the roots of $g(z)$ inside the unit disk) 
with approximated coefficients $\tilde{b}_k$ such that $|\tilde{b}_k-b_k|$ is $O(\eps)$ as $\eps\rightarrow0$.
\end{cor}
\begin{proof}
First, we note that $\frac{1}{z}=\overline{z}$ on $\S1$. Thus, $|g(z)|=|z^{d-1}f(\overline{z})|=|f(\overline{z})|$ on $\S1$. Then because we have $m\le|f(\overline{z})|\le M$ on $\S1$, we also have $m\le|g(z)|\le M$ on $\S1$. We also know that $g'(z)=(d-1)\frac{1}{z}g(z)-z^{d-1-2}f'(\frac{1}{z})$, so that $|g'(z)|\le (d-1)|g(z)|+|f'(\overline{z})|\le(d-1)M+M'$ on $\S1$. Note that on $\S1$, $|g(z)|^2=|f(\overline{z})|^2$, which has a perturbation of $p_2(\overline{z})$, and
\begin{align*}
g'(z)\overline{g(z)}
&=\left((d-1)\frac{1}{z}g(z)-z^{d-1-2}f'(\tfrac{1}{z})\right)\overline{g(z)}\\
&=(d-1)\overline{z}|g(z)|^2-z^{d-1-2}f'(\overline{z})\overline{z^{d-1}f(\overline{z})}\\
&=(d-1)\overline{z}|f(\overline{z})|^2-\overline{z}^2f'(\overline{z})\overline{f(\overline{z})}
\end{align*}
which has a perturbation of $(d-1)\overline{z}p_2(\overline{z})-\overline{z}^2p_1(\overline{z})$. Note that both of these perturbations are bounded by $d\eps$. Thus, $g(z)$ is a complex polynomial that satisfies the requirements of the theorem, with $N$ replaced by $d$, $M'$ replaced by $(d-1)M+M'$, $\eps$ replaced by $d\eps$, and $\alpha$ replaced by $\beta$. Thus by the theorem, using the perturbed functions for $f(z)$, we can recover an approximation for $g_i(z)$ with approximated coefficients $\tilde{b}_k$ such that $|\tilde{b}_k-b_k|$ is $O(d\eps)=O(\eps)$.
\end{proof}

The objective of the following proposition is to control the error when the recovery of the polynomial factors
from the inner and the outer roots are combined.

\begin{prop}
If $f$ satisfies the hypotheses of the preceding theorem, and in addition $\max_{z \in \S1} |f(z)| = M$ and $\eps<\frac{\beta m^2}{d}$ for $d>N$ with
$$\beta=\frac{1}{1+2\left(1+\frac{(d-1)M+M'}{m}\right)}$$
then using the perturbed functions for $f(z)$, we can recover an approximation (up to a multiplicative constant) for $f(z)=\sum_{k=0}^{d-1}c_kz^k$
 with approximated coefficients $\tilde{c}_k$ such that $\max_k |\tilde{c}_k-c_k|$ is $O(\eps)$.
\end{prop}
\begin{proof}
Let $N_i$ be the number of roots of $f(z)$ inside the unit disk, and let $N_o$ be the number of roots of $f(z)$ outside the unit disk. Note that $g_i(z)$ obtained in the previous corollary has roots $(\frac{1}{z_j})_{j=1}^{N_o}$ for all roots $z_j$ of $f(z)$ outside the unit disk, with $d-1-N$ additional roots at $0$. Thus,
$$g_i(z)=z^{d-1-N}\prod_{j=1}^{N_o}\left(z-\frac{1}{z_j}\right)$$
Then if we let $f_o(z)=z^{d-1-N_i}g_i(\frac{1}{z})$, we get
$$f_o(z)=z^{d-1-N_i}z^{N-(d-1)}\prod_{j=1}^{N_o}\left(\frac{1}{z}-\frac{1}{z_j}\right)=z^{N-N_i-N_o}\prod_{j=1}^{N_o}\frac{1}{z_j}\left(z_j-z\right)=\prod_{j=1}^{N_o}\frac{(-1)^{N_o}}{z_j}\left(z-z_j\right)$$
and so if $(z_j')_{j=1}^{N_i}$ are the roots of $f(z)$ inside the unit disk, by applying the thoerem we get
$$f_o(z)f_i(z)=\prod_{j=1}^{N_o}\frac{(-1)^{N_o}}{z_j}\left(z-z_j\right)\prod_{j=1}^{N_i}(z-z_j')$$
which is a constant multiple of $f(z)$. Thus, we just need an approximation for $f_o(z)f_i(z)$. In terms of the coefficients of $g_i(z)$, we have
$$f_o(z)=z^{d-1-N_i}\sum_{k=0}^{N_0+d-1-N}b_k\frac{1}{z^k}=\sum_{k=0}^{d-1-N_i}b_kz^{d-1-N_i-k}=\sum_{k=0}^{d-1-N_i}b_{d-1-N_i-k}z^k$$
Note that if $r$ is the multiplicity of the $0$ root of $g_i(z)$, then for all $k$ from $0$ to $r-1$, $b_k=0$. Thus, the coefficients for the $r$ highest degree terms of $f_o(z)$ are equal to $0$, and $f_o(z)$ has degree $N_o$ as it should. Note that we can obtain $f_o(z)$ simply by reversing the order of the coefficients of $g_i(z)$, so our approximated coefficients $\tilde{b}_k$ for $f_o(z)$ are the same approximated coefficients that we obtained in the previous corollary, but in reverse order. Thus $|\tilde{b}_k-b_k|$ is $O(\eps)$ as $\eps\rightarrow0$. In other words, there are numbers $C_{d,b_k}$ which do not depend on $\eps$, such that $|\tilde{b}_k-b_k|\le C_{d,b_k}\eps$. Also, from the theorem, we have that the approximated coefficients $\tilde{a}_j$ for $f_i(z)=\sum_{j=0}^{N_i}a_jz^j$ also have error that is $O(\eps)$ as $\eps\rightarrow0$, so there are numbers $C_{d,a_j}$ which do not depend on $\eps$, such that $|\tilde{a}_j-a_j|\le C_{d,a_j}\eps$. Note that on $\S1$
$$|f_i(z)|=\prod_{j=1}^{N_i}|z-z_j'|\le2^{N_i}$$
and
$$|f_o(z)|=\left|z^{d-1-N_i}z^{N-(d-1)}\prod_{j=1}^{N_o}\left(\frac{1}{z}-\frac{1}{z_j}\right)\right|=\prod_{j=1}^{N_o}\left|\overline{z}-\frac{1}{z_j}\right|\le2^{N_o}, z\in \S1$$
Thus, the max norms on $\S1$ of $f_i(z)$ and $f_o(z)$ are bounded by $2^{N_i}$ and $2^{N_o}$ respectively. Because all norms on a finite dimensional space are equivalent, there exists a number $K_N$, which non-decreasingly depends only on $N$, such that for any complex polynomial $h(z)=\sum_{k=0}^Nc_kz^k$ of degree less than or equal to $N$, we have $\max_{k\le N}|c_k|\le K_N\|h(z)\|_\infty$. Note that
$$f(z)=f_o(z)f_i(z)=\left(\sum_{k=0}^{d-1-N_i}b_{d-1-N_i-k}z^k\right)\left(\sum_{j=0}^{N_i}a_jz^j\right)=\sum_{n=0}^{d-1}\left(\sum_{k=0}^nb_{d-1-N_i-k}a_{n-k}\right)z^n$$
and thus, for the approximation $\tilde{f}(z)=\sum_{k=0}^{d-1}\tilde{c}_kz^k$ we have that for each $k$
$$\tilde{c}_k=\sum_{j=0}^n\tilde{b}_{d-1-N_i-j}\tilde{a}_{k-j}$$
Then for each $k$ we have
\begin{align*}
|\tilde{c}_k-c_k|
&=\left|\sum_{j=0}^n\tilde{b}_{d-1-N_i-j}\tilde{a}_{k-j}-\sum_{j=0}^n b_{d-1-N_i-j}a_{k-j}\right|\\
&\le\sum_{j=0}^n\left|\tilde{b}_{d-1-N_i-j}\tilde{a}_{k-j}-b_{d-1-N_i-j}a_{k-j}\right|\\
&\le\sum_{j=0}^n\left(\left|\tilde{b}_{d-1-N_i-j}\tilde{a}_{k-j}-\tilde{b}_{d-1-N_i-j}a_{k-j}\right|+\left|\tilde{b}_{d-1-N_i-j}a_{k-j}-b_{d-1-N_i-j}a_{k-j}\right|\right)\\
&=\sum_{j=0}^n\left(|\tilde{b}_{d-1-N_i-j}|\left|\tilde{a}_{k-j}-a_{k-j}\right|+\left|\tilde{b}_{d-1-N_i-j}-b_{d-1-N_i-j}\right||a_{k-j}|\right)\\
&\le\sum_{j=0}^n\left((K_{N_o}2^{N_o}+C_{d,b_{d-1-N_i-j}}\eps)C_{d,a_{k-j}}\eps+C_{d,b_{d-1-N_i-j}}\eps K_{N_i}2^{N_i}\right)\\
&\le K_d2^d\sum_{j=0}^n\left(C_{d,a_{k-j}}+C_{d,b_{d-1-N_i-j}}\right)\eps+O(\eps^2)
\end{align*}
Thus, since $\eps$ was bounded above, multiplying two polynomials, whose coefficients have an error that is $O(\eps)$ gives a polynomial whose coefficients have an error that is $O(\eps)$.
\end{proof}

Finally, we obtain a finite number of measurements by discretizing and interpolating with the Dirichlet kernel.

\begin{thm}
Let $f(z)=\sum_{k=0}^{d-1}c_kz^k$ be a complex polynomial with degree at most $d-1$, with fixed positive constants $m$, $M$ and $M'$ such that $m\le|f(z)|\le M$ and $|f'(z)|\le M'$ for all $z$ on the unit circle $\S1$. Let $\omega=e^\frac{2\pi i}{2d-1}$ and $\nu=e^\frac{2\pi i}{3}$ be the $(2d-1)$th and $3$rd roots of unity, and let
$$\beta=\frac{1}{1+2\left(1+\frac{(d-1)M+M'}{m}\right)}$$
Let $\eps>0$ with $\eps<\frac{\beta m^2}{(2d-1)d}$, and assume that we know $2d-1$ values $\{|f(\omega^l)|^2+\eps_{l,0}\}_{l=0}^{2(d-1)}$ and $6d-3$ values $\{|f(\omega^l)+\nu^jf'(\omega^l)|^2+\eps_{l,j}\}_{l=0}^{2(d-1)}$ $_{j=1}^3$ with each $\eps_{l,j}\le\eps$. Then using only these values, we can recover an approximation (up to a multiplicative constant) for $f(z)$ with approximated coefficients $\tilde{c}_k$ such that $|\tilde{c}_k-c_k|$ is $O(\eps)$.
\end{thm}
\begin{proof}
If $D_{d-1}(z)=\frac{1}{2d-1} \sum_{k=-(d-1)}^{d-1}z^k$ is the normalized Dirichlet kernel of degree $d-1$, then the set of functions $\{z \mapsto D_{d-1}(z\omega^l)\}_{l=0}^{2(d-1)}$ is orthogonal with respect to the $L^2$ norm on $\S1$, and in addition it provides interpolation identity for each trigonometric polynomial  $g(z)=\sum_{k=-(d-1)}^{d-1}c_kz^k$, $g(z)=\sum_{l=0}^{2(d-1)}g(\omega^l)D_{d-1}(z\omega^{-l})$. If we let $g_0(z)=|f(z)|^2$ and $g_j(z)=|f(z)+\nu^jf'(z)|^2$ for $j=1,2,3$, then because each of these is a trigonometric polynomial
of degree at most $d-1$, we know that  we have
\begin{align*}
\sum_{l=0}^{2(d-1)}\left(g_j(\omega^l)+\eps_{l,j}\right)D_{d-1}(z\omega^{-l})
&=\sum_{l=0}^{2(d-1)}g_j(\omega^l)D_{d-1}(z\omega^{-l})+\sum_{l=0}^{2(d-1)}\eps_{l,j}D_{d-1}(z\omega^{-l})\\
&=g_j(z)+\sum_{l=0}^{2(d-1)}\eps_{l,j}D_{d-1}(z\omega^{-l})
\end{align*}
Let $h_j(z)=\sum_{l=0}^{2(d-1)}\eps_{l,j}D_{d-1}(z\omega^{-l})$ be the error obtained when using the Dirichlet kernel with the known given values to recover each $g_j(z)$. 
Note that
$$|h_j(z)|
=\left|\sum_{l=0}^{2(d-1)}\eps_{l,j}D_{d-1}(z\omega^{-l})\right|
\le\sum_{l=0}^{2(d-1)}\left|\eps_{l,j}D_{d-1}(z\omega^{-l})\right|
\le(2d-1)\eps|D_{d-1}(z\omega^{-l})|=(2d-1)\eps$$
Thus we have recovered approximations for the functions $g_j(z)$ on $\S1$, including $g_0(z)=|f(z)|^2$ with perturbation that is less than $(2d-1)\eps<(2d-1)\frac{\beta m^2}{(2d-1)d}=\frac{\beta m^2}{d}$. However, to use the previous corollary, we also need an approximation for $f'(z)\overline{f(z)}$. To obtain this, note that
\begin{align*}
\frac{1}{3}\sum_{j=1}^3\overline{\nu}^j(g_j(z)+h_j(z))
&=\frac{1}{3}\sum_{j=1}^3\overline{\nu}^jg_j(z)+\frac{1}{3}\sum_{j=1}^3\overline{\nu}^jh_j(z)\\
&=\frac{1}{3}\sum_{j=1}^3\overline{\nu}^j\left|f(z)+\nu^jf'(z)\right|^2+\frac{1}{3}\sum_{j=1}^3\overline{\nu}^jh_j(z)\\
&=\frac{1}{3}\sum_{j=1}^3\overline{\nu}^j\left(f(z)+\nu^jf'(z)\right)\left(\overline{f(z)+\nu^jf'(z)}\right)+\frac{1}{3}\sum_{j=1}^3\overline{\nu}^jh_j(z)\\
&=\frac{1}{3}\sum_{j=1}^3\left(\overline{\nu}^j|f(z)|^2+\overline{\nu}^{2j}f(z)\overline{f'(z)}+f'(z)\overline{f(z)}+\overline{\nu}^j|f'(z)|^2\right)\\
&\qquad\qquad\qquad\qquad+\frac{1}{3}\sum_{j=1}^3\overline{\nu}^jh_j(z)\\
&=f'(z)\overline{f(z)}+\frac{1}{3}\sum_{j=1}^3\overline{\nu}^jh_j(z)
\end{align*}
Note that $\left|\frac{1}{3}\sum_{j=1}^3\overline{\nu}^jh_j(z)\right|\le\frac{1}{3}\sum_{j=1}^3|h_j(z)|<\frac{\beta m^2}{d}$. Thus, we have recovered approximations for $|f(z)|^2$ and $f'(z)\overline{f(z)}$ with error bounded by $\frac{\beta m^2}{d}$. Then, by using the previous corollary with $(2d-1)\eps$ in place of the $\eps$ in that corollary, we recover an approximation (up to a multiplicative constant) for $f(z)$ with approximated coefficients $\tilde{c}_k$ such that $|\tilde{c}_k-c_k|$ is $O(\eps)$.
\end{proof}

\begin{cor}
Let $f$ be a complex polynomial of degree at most $d-1$, let $\omega$ and $\nu$ and $\eps>0$ satisfy the conditions with $m$, $M$ and $\beta$ be as in the preceding theorem,
then 
 $\{|f(\omega^l)|^2+\eps_{l,0}\}_{l=0}^{2(d-1)}\cup\{|f(\omega^l)+\nu^jf'(\omega^l)|^2+\eps_{l,j}\}_{l=0}^{2(d-1)}$ $_{j=1}^3$ with each $\eps_{l,j}\le\eps$
 determines an approximation of $f$, up to a unimodular constant, with accuracy $O(\epsilon)$.
\end{cor}
\begin{proof}
The measured quantities $\{|f(\omega^l)|^2+\eps_{l,0}\}_{l=0}^{2(d-1)}$, by the Parseval identity, determine $\|f\|_2^2$, up to
an error proportional to $\epsilon$. By the point-wise lower bound on $|f|$, the norm is bounded below by $\|f\|_2 \ge (2\pi m)^{1/2}$,
so $\|f\|_2$ is also known with accuracy $O(\epsilon)$.
Let $C=(\sum_{k=0}^{d-1} |\tilde b_k |^2)^{1/2}$ then $\tilde c_k = \|f\|_2 \tilde b_k / C$ determines an approximation
$g(z)=\sum_{k=0}^{d-1}\tilde c_k z^k$ such that $\|f/g-c\|_{\infty}$ is $O(\eps)$ for some $c\in \mathbb C$ with $|c|=1$.
\end{proof}


To illustrate these results we ran simulations of perturbed values to verify the recovery procedure. We randomly generated a polynomial, and ran a large number of trials with randomized $\eps$ perturbations on the values needed to apply the theorem. For all sample polynomials, we obtained results showing a linear relation between the perturbation $\eps$ and the difference in the coefficients of the recovered polynomial and the original polynomial. 
%
Even when we perturb the values by 50 times the radius of stability $\eps_0=\frac{\beta m^2}{(2d-1)d}$ given by the theorem, the linear bound remains intact. This indicates that either the radius of stability used in our theorem is not the sharpest value that we can obtain for this, or that unstable behavior happens outside of this radius only for pathological examples.
\begin{figure}[h]
\centering
\includegraphics[scale=.9]{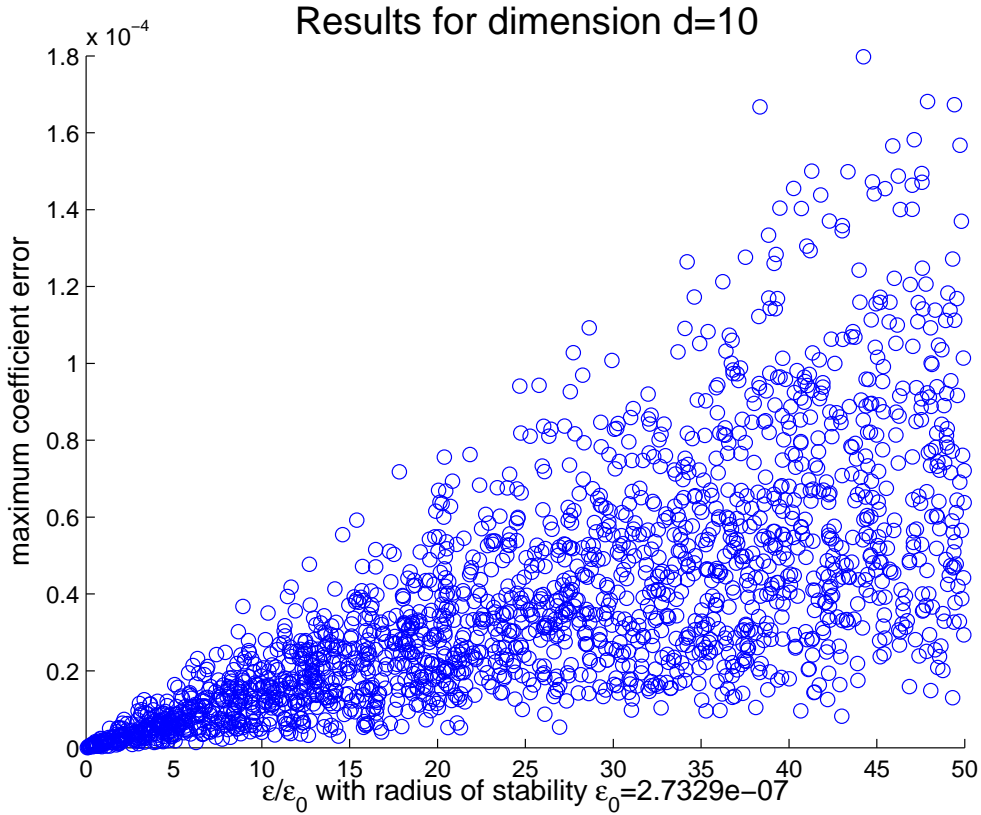}
\caption{Experimentally found recovery error from applying the Newton identities to perturbed measurements.
This plot shows that the linear bound for the recovery error in terms of the noise level remains intact even for values of $\eps$ that are much larger than $\epsilon_0$. }
\label{fig:high}
\end{figure}

\bibliographystyle{amsalpha}

\end{document}